\documentclass[12pt,reqno]{amsart}
\usepackage{amsmath,amssymb}
\usepackage{enumitem}
\usepackage{graphics,color}
\usepackage{epsfig}
\usepackage{graphicx}
\usepackage{amsfonts,amsmath,amssymb,amscd}
\usepackage{xcolor}
\usepackage{hyperref}

\definecolor{ao(english)}{rgb}{0.0, 0.5, 0.0}
\definecolor{pink}{rgb}{1.0, 0.0, 1.0}

\hypersetup{colorlinks=true, linkcolor=ao(english),citecolor=ao(english)}

\setbox0=\hbox{$+$}
\newdimen\plusheight
\plusheight=\ht0
\def\+{\;\lower\plusheight\hbox{$+$}\;}

\setbox0=\hbox{$-$}
\newdimen\minusheight
\minusheight=\ht0
\def\-{\;\lower\minusheight\hbox{$-$}\;}

\setbox0=\hbox{$\cdots$}
\newdimen\cdotsheight
\cdotsheight=\plusheight
\def\cds{\lower\cdotsheight\hbox{$\cdots$}}

\newcommand{\Q}{\mbox{$\mathbb Q$}}
\renewcommand{\(}{\left\(}
\renewcommand{\)}{\right\)}
\renewcommand{\[}{\left[}

\renewcommand{\pmod}[1]{\,(\textup{mod}\,#1)}
\theoremstyle{plain}
\newtheorem{theorem}{Theorem}[section]
\newtheorem{lemma}[theorem]{Lemma}

\newtheorem{conjecture}[theorem]{Conjecture}
\theoremstyle{definition}

\theoremstyle{remark}

\numberwithin{equation}{section}

\begin{document}
\title[Proof of Two Conjectural $q$-Supercongruences]{Proof of Two Conjectural $q$-Supercongruences Related to Van Hamme's $\mathrm{(G.2)}$ Supercongruence}
\author{Liton Karmakar and Arijit Jana}
\address{Department of Mathematics, National Institute of Technology Silchar, Assam 788010, India}
\email{litonofficial8638@gmail.com}

\address{Department of Mathematics, National Institute of Technology Silchar, Assam 788010, India}
\email{jana94arijit@gmail.com}

\begin{abstract}
In this paper, we prove two conjectures concerning parametric $q$- supercongruences related to Van Hamme's $\mathrm{(G.2)}$ supercongruence. First, we confirm Tang\rq s conjecture \cite[Conjecture $4.1$]{Tang G2} by establishing the corresponding congruences modulo $[m]\Phi_m(q)^2$. Second, we prove a conjecture of Wang and Xu \cite[Conjecture $2$]{WangXu24}, extending their $q$-supercongruence modulo $\Phi_m(q)^5$ to modulo $[m]\Phi_m(q)^4$ when $m\equiv1\pmod4$.
\end{abstract}

\maketitle
\section{Introduction and statement of results}
In $1913$, Ramanujan \cite{berndt-rankin} mentioned the following identity in his first letter to Hardy:
\begin{equation}\label{Ramanujan G1}
\mathrm{(G.1)}\qquad
\sum_{n=0}^{\infty} (8n+1)
\frac{\left(\frac{1}{4}\right)_n^4}{n!^4}
=
\frac{2\sqrt{2}}{\sqrt{\pi}\,\Gamma\left(\frac{3}{4}\right)^2},
\end{equation}
where $\Gamma(x)$ denotes the Gamma function and $(x)_n=x(x+1)\cdots(x+n-1)$ denotes the rising factorial. Ramanujan did not provide a proof of the identity \eqref{Ramanujan G1} in his letter to Hardy, and Hardy \cite{hardy} was the first to prove the infinite series identity \eqref{Ramanujan G1}. In $1997$, Van Hamme \cite{van2020some} listed several conjectural supercongruences connecting Ramanujan's formulas for $\frac{1}{\pi}$ with the $p$-adic Gamma function. The elegant $p$-adic analogue of \eqref{Ramanujan G1} is given by
\begin{equation}\label{padic G2}
\mathrm{(G.2)}\qquad
\sum_{n=0}^{\frac{p-1}{4}} (8n+1)
\frac{\left(\frac{1}{4}\right)_n^4}{n!^4}
\equiv
p\frac{\Gamma_p\left(\frac{1}{2}\right)
\Gamma_p\left(\frac{1}{4}\right)}
{\Gamma_p\left(\frac{3}{4}\right)}
\pmod{p^3},
\qquad \text{if } p\equiv 1\pmod{4}.
\end{equation}
Here, $p$ is an odd prime, and $\Gamma_p(x)$ denotes the $p$-adic Gamma function. The supercongruence \eqref{padic G2} was subsequently extended to the modulus $p^4$ independently by Swisher \cite{swisher} and He \cite{he}. In $2023$, Tang \cite{Tang G2} established the following $q$-supercongruence, which provides a parametric generalization of the congruence \eqref{padic G2}.
\begin{theorem} \cite[Theorem 1.1]{Tang G2} \label{Tang G2 Thm 1}
Let $m \equiv 1 \pmod{4}$ be an integer greater than $1$ and let $0 \leq s \leq \frac{m-1}{4}$. Then
\begin{align} \label{Tang G2 Thm 1 eqn}
\sum_{n=s}^{ \frac{m-1}{4}}
& [8n+1]
\frac{(q;q^4)_{n-s}(q;q^4)_{n+s}(q;q^4)_n^2}
{(q^4;q^4)_{n-s}(q^4;q^4)_{n+s}(q^4;q^4)_n^2}
q^{2n} \notag \\
\equiv &
[m+4s]\,
\frac{(q;q^4)_s^2(q;q^4)_{\frac{m-1}{4}+s}(q^2;q^4)_{\frac{m-1}{4}-s}}
{(q^4;q^4)_s^2(q^4;q^4)_{\frac{m-1}{4}+s}(q;q^4)_{\frac{m-1}{4}-s}}
q^{3s-(\frac{m-1}{4})} \pmod{\Phi_m(q)^3}.
\end{align} 
\end{theorem}
At the end of the paper, Tang posed the following open problem.
\begin{conjecture} \cite[Conjecture $4.1$]{Tang G2} \label{Conj Tang G2}
    The $q$-supercongruence \eqref{Tang G2 Thm 1 eqn} is also true modulo $[m]\Phi_m(q)^2.$
\end{conjecture}
\par To continue the $q$-story, we first recall some standard $q$-notation. The \emph{$q$-shifted factorial} for complex numbers $x$ and $q$, with $|q|<1$, is defined by
\begin{align*}
(x;q)_0 &:= 1,\\
(x;q)_n &:= \prod_{k=0}^{n-1}(1-xq^k), \qquad n\geq 1.
\end{align*}
The $q$-integer is defined by
$$[m]=[m]_q:=1+q+q^2+\cdots+q^{m-1}
=\frac{1-q^m}{1-q}.$$
Moreover, the $m$-th cyclotomic polynomial $\Phi_m(q)$ is defined by
$$\Phi_m(q):=
\prod_{\substack{1\leq d\leq m\\ \gcd(d,m)=1}}
\left(q-\xi^d\right),$$
where $\xi$ is a primitive $m$-th root of unity. Rational function congruences are interpreted in the standard reduced-denominator sense. More precisely, if
$R(q)=\frac{A(q)}{B(q)},$ with $ A(q),B(q)\in\mathbb{Q}[q], \gcd(A(q),B(q))=1,$ and $B(q)$ is relatively prime to the given modulus, then
$R(q)\equiv 0\pmod{M(q)}$ means that $M(q)$ divides $A(q)$ in $\mathbb{Q}[q]$.


\par Recently, Wang and Xu \cite{WangXu24} established several new $q$-supercongruences. In particular, they proved the following $q$-supercongruence modulo the fifth power of a cyclotomic polynomial.

\begin{theorem} \label{WangXu Thm}
     Let $m>1$ be an odd integer, $0\le s\le (m-t)/4$ and
$t\in\{1,3\}$ such that $m \equiv t \pmod 4$. Then, modulo
$\Phi_m(q)^5$,
\begin{align} \label{WangXu Thm eqn}
&\sum_{n=s}^{(tm-1)/4}
 [8n+1]
\frac{(q;q^4)_{n-s}(q;q^4)_{n+s}(q;q^4)_n^{4}}
{(q^4;q^4)_{n-s}(q^4;q^4)_{n+s}(q^4;q^4)_n^{4}}
q^{5n} \notag
\\
& \equiv
[tm+4s] \frac{(q;q^4)_s^4 (q^2;q^4)_{\frac{tm-1}{4}-s} (q;q^4)_{\frac{tm-1}{4}+s}}{(q^4;q^4)_s^4 (q;q^4)_{\frac{tm-1}{4}-s} (q^4;q^4)_{\frac{tm-1}{4}+s}} q^{\frac{1-tm}{4} + 6s} \sum_{n=0}^{\frac{tm-1}{4}-s} \frac{(q,q^3,q^{4s+1},q^{4s+1};q^4)_n}{(q^2,q^4,q^{4s+4},q^{4s+4};q^4)_n}q^{4n} \notag
\\
& \qquad \times \left\{1 - (2-q^{tm}) [tm]^2 \sum_{j=1}^{n} \left(\frac{q^{4(s+j)-3}}{[4(s+j)-3]^2} +  \frac{q^{4(s+j)}}{[4(s+j)]^2}\right)   \right\}.
\end{align}
\end{theorem}
At the end of their paper, they posed the following open problem.
\begin{conjecture} \cite[Conjecture $2$]{WangXu24} \label{WangXu Conj}
    Theorem \ref{WangXu Thm} with $t=1$ is true modulo $[m]\Phi_m(q)^4.$
\end{conjecture}
It is worth mentioning that Conjecture \ref{Conj Tang G2} is known to hold in the case $s=0$; see \cite{GuoMJS20, guo2019q}. Based on numerical evidence, Tang observed that both sides of \eqref{Tang G2 Thm 1 eqn} should be congruent to $0$ modulo $[m]$, although proving this appears to be difficult. Tang further remarked that the well-known method for establishing $q$-congruences modulo $[m]$, introduced in \cite{guo2019q}, does not seem to be applicable in this setting due to the presence of the additional parameter $s$. Very recently, Liu and Tao \cite{liu2026proofs} confirmed a $q$-supercongruence conjectured by Guo \cite[Conjecture 4.1]{guo2022new}. Motivated by their work, we establish the following two theorems.
\begin{theorem} \label{P6 Main Thm 1}
    Conjecture \ref{Conj Tang G2} is true.
\end{theorem}
\begin{theorem} \label{P6 Main Thm 2}
    Conjecture \ref{WangXu Conj} is true.
\end{theorem}
\section{Key lemmas } \label{Sect 2}
In this section, we establish several auxiliary lemmas that will be used in the proofs of our main results. First we define
\begin{align*}
    \Theta_q(n,s) = [8n+1]
\frac{(q;q^4)_{n-s}(q;q^4)_{n+s}(q;q^4)_n^{2}}
{(q^4;q^4)_{n-s}(q^4;q^4)_{n+s}(q^4;q^4)_n^{2}}
q^{2n},
\end{align*}
\begin{align*}
    \Xi_q(n,s)= [8n+1]
\frac{(q;q^4)_{n-s}(q;q^4)_{n+s}(q;q^4)_n^{4}}
{(q^4;q^4)_{n-s}(q^4;q^4)_{n+s}(q^4;q^4)_n^{4}}
q^{5n}.
\end{align*}
Furthermore, set
\begin{align*}
    \Omega_q^1(\mathcal{T}) = \sum_{n=s}^{\mathcal{T}} \Theta_q(n,s),
\end{align*}
and 
\begin{align*}
    \Omega_q^2(\mathcal{T}) = \sum_{n=s}^{\mathcal{T}} \Xi_q(n,s),
\end{align*}
where $\mathcal{T}= (m-1)/4$. 

\par \quad Throughout the paper, let $d=4g+1>1$ be an odd integer, and let $\xi$ denote a primitive $d$-th root of unity. For $M\ge0$, define
\begin{equation*}
  H_M(q)=\frac{(q;q^4)_M}{(q^4;q^4)_M},
\end{equation*}
and
\begin{equation*}
  J_a=\prod_{j=0}^{a-1}\frac{4j+1}{4j+4}=\frac{(1/4)_a}{(1)_a}, \qquad a\ge 0.
\end{equation*}

\begin{lemma} \label{P6 lemma 1}
    Let $M=ad+b$, where $a\ge0$ and $0\le b<d$, then  
\begin{equation}\label{P6 lemma 1 eq}
H_M(\xi)=
\begin{cases}
J_a H_b(\xi),&0\le b\le g,\\
0,& g<b<d,
\end{cases}
\end{equation}
when $q\to\xi$. In particular, this has a finite limit at $\xi$.
\end{lemma}
\begin{proof}
The product that defines $H_M$ may be seen in a whole block of length $d$, represented as $i=jd,jd+1,\ldots,jd+d-1$. Since $4 (jd+g) + 1 = (4j+1)d,$ the unique numerator zero happens at $i=jd+g.$ Similarly, since $4(jd+d-1) + 4 = (4j+4)d,$ the unique zero factor in the denominator occurs at $i=jd+d-1$. Consequently, the ratio of the corresponding zero factors is
\begin{equation*} 
\lim_{q\to\xi}\frac{1-q^{(4j+1)d}}{1-q^{(4j+4)d}} = \frac{4j+1}{4j+4}.
\end{equation*}
  The exponents in the remaining numerator and denominator would run through the same non-zero residue classes modulo $d$ if these two zero factors were eliminated. This implies that the remaining products would cancel at $q=\xi$. Consequently, the contribution of the $j$-th full block is $(4j+1)/(4j+4)$. 
\par\quad The remaining incomplete block has length $b.$\\
\textbf{Case I:} When $0 \le b \le g,$ it contains no zero factors, and therefore behaves precisely as $H_b(\xi).$ \\
\textbf{Case II:} When $g < b < d,$ it contains a zero factor in the numerator, whereas the denominator has no corresponding zero factor. Consequently, the entire product vanishes. This completes the proof.
\end{proof}
The next lemma establishes a reflection identity for $ H_M$.

\begin{lemma} \label{P6 lemma 2}
    For $0\le x\le g$, we have
\begin{equation}\label{P6 lemma 2 eq}
  H_{g-x}(\xi)=(-1)^g\xi^{2g^2-g+3x}H_x(\xi).
\end{equation}
Moreover, for $0\le r\le g$,
\begin{equation*}
  [8(g-r)+1]_{\xi}=-\xi^{-(8r+1)}[8r+1]_{\xi}.
\end{equation*}
\end{lemma}
\begin{proof}
  Since none of the denominator factors in $H_x(\xi)$ and $H_{g-x}(\xi)$ vanishes, we define

$$Q_x(a,q):=\frac{(aq;q^4)_x}{(q^4/a;q^4)_x}.$$

\textbf{Claim :} 
\begin{equation}\label{P5 P-reflection}
  Q_{g-x}(a,q)\equiv
  (-a)^{g-2x}q^{2g^2-g+3x}Q_x(a,q)
  \pmod{\Phi_d(q)}.
\end{equation}
\textbf{Case I :}
When $0\leq x\leq g/2$, the denominator factors can, modulo $\Phi_d(q)$, be reversed by replacing $j$ with $g-1-j$:

\begin{align*}
\frac{Q_{g-x}(a,q)}{Q_x(a,q)}
&=\prod_{j=x}^{g-x-1}\frac{1-aq^{4j+1}}{1-q^{4j+4}/a}  \\
&\equiv
  \prod_{j=x}^{g-x-1}\frac{1-aq^{4j+1}}{1-q^{-(4j+1)}/a} \\
&=\prod_{j=x}^{g-x-1}(-a q^{4j+1})                         \\
&=(-a)^{g-2x}q^{2g^2-g-4xg+2x}
 \equiv (-a)^{g-2x}q^{2g^2-g+3x},
\end{align*}
because $(2g^2-g-4xg+2x) - (2g^2-g+3x) = -xd.$ \\
\textbf{Case II :} 
When $x>g/2$, we replace $x$ by $g-x$. Then, by \textbf{Case I}, we obtain

\begin{align*}
    \frac{Q_{x}(a,q)}{Q_{g-x}(a,q)} & \equiv (-a)^{2x-g}q^{2g^2+2g-3x} \pmod{\Phi_d(q)}.
\end{align*}
Taking reciprocals, we obtain
\begin{align*}
    \frac{Q_{g-x}(a,q)}{Q_{x}(a,q)} & \equiv (-a)^{g-2x}q^{-(2g^2+2g-3x)} \pmod{\Phi_d(q)} \\
    & \equiv (-a)^{g-2x}q^{2g^2-g+3x} \pmod{\Phi_d(q)},
\end{align*}
because $(2g^2+2g-3x) + (2g^2-g+3x)=gd.$ This proves the claim in \eqref{P5 P-reflection}. Setting $a=1$ and $q=\xi$ then completes the proof of \eqref{P6 lemma 2 eq}.
 Since $8(g-r)+1=2d-(8r+1),$ we obtain
\begin{equation*}
[8(g-r)+1]_{\xi}
=\frac{1-\xi^{-(8r+1)}}{1-\xi}
=-\xi^{-(8r+1)}[8r+1]_{\xi}.
\end{equation*}
This completes the proof of the lemma.
\end{proof}

\begin{lemma} \cite[Lemma $3.3$]{liu2026proofs} \label{P6 lemma 3}
    Let $d>1$, and let $R(q)=A(q)/B(q)\in\Q(q)$, where $A,B\in\Q[q]$ are coprime. If $R(q)$ has a finite limit at every primitive $d$-th root of unity, then $\Phi_d(q)\nmid B(q)$.  If, in addition, $R(\xi)=0$ for every primitive $d$-th root of unity $\xi$, then $\Phi_d(q)\mid A(q)$.
\end{lemma}

Throughout the paper, let $d=4g+1>1$ be a fixed divisor of $m$. We express $s$ in the form
\begin{equation}\label{P5 eq 's'}
s=ud+v,
\qquad 0\le v<d.
\end{equation}

For an integer $y$, let $\langle y\rangle_d$ denote the least non-negative residue of $y$ modulo $d$. By Lemma \ref{P6 lemma 1}, $H_M(\xi)$ vanishes whenever $\langle M\rangle_d>g$. Consequently, we have
\begin{equation}\label{P6 T-nonzero-condition 1}
\Theta_{\xi}(n,s)=0,
\end{equation}
and
\begin{equation}\label{P6 T-nonzero-condition 2}
\Xi_{\xi}(n,s)=0,
\end{equation}
unless
\begin{equation}\label{P5 three-residue-conditions}
\langle n-s\rangle_d\leq g,
\qquad
\langle n\rangle_d\leq g,
\qquad
\langle n+s\rangle_d\leq g.
\end{equation}

Now we set $r=\langle n\rangle_d$, then, since $s\equiv v\pmod d$, we have from the above conditions \eqref{P5 three-residue-conditions}, 
\begin{equation}\label{P5 r-residue-conditions}
  \langle r-v\rangle_d\leq g,\qquad
  r\leq g,\qquad
  \langle r+v\rangle_d\leq g .
\end{equation}
We claim that the above conditions are equivalent to $r\in \Psi_v$, where
\begin{equation}\label{P5 Rb-def}
\Psi_v=
\begin{cases}
\{v,v+1,\ldots,g-v\},&
0\le v\le \lfloor g/2\rfloor,\\[2mm]
\{e,e+1,\ldots,g-e\},&
e=d-v,\ 1\le e\le \lfloor g/2\rfloor,\\[2mm]
\varnothing,&\text{otherwise}.
\end{cases}
\end{equation}
To prove the above claim, we consider the following two cases.\\
\textbf{Case A :}
When $0\leq v\leq g$, the middle condition in \eqref{P5 r-residue-conditions} yields $0\leq r\leq g$. If $r\geq v$, then clearly $\langle r-v\rangle_d\leq g$. On the other hand, if $r<v$, then
$$
\langle r-v\rangle_d=d+r-v\geq d-v=3g+1+(g-v)>g,
$$
which contradicts \eqref{P5 r-residue-conditions}. Therefore, we must have $r\geq v$, and hence $\langle r-v\rangle_d =r-v\leq g$.
Similarly, we obtain $r\leq g-v$. Indeed, since $0\leq r+v\leq 2g<4g+1=d$, the least non-negative residue of $r+v$ modulo $d$ is simply $r+v$. Therefore,
$$
\langle r+v\rangle_d\leq g
\quad\Longleftrightarrow\quad
r+v\leq g
\quad\Longleftrightarrow\quad
r\leq g-v.
$$
Combining this with $r\geq v$, we obtain
$$
v\leq r\leq g-v.
$$
Thus, the three conditions in \eqref{P5 r-residue-conditions} are equivalent to $v\leq r\leq g-v$. This gives the first line of \eqref{P5 Rb-def}. In particular, if $v>\lfloor g/2\rfloor$, the corresponding interval is empty.\\
\textbf{Case B :}
When $v>g$, write $v=d-e$, where $1\leq e\leq g$. Then $v\equiv -e\pmod d$. Hence, the conditions in \eqref{P5 r-residue-conditions} become
$$
\langle r+e\rangle_d\leq g,\qquad
r\leq g,\qquad
\langle r-e\rangle_d\leq g.
$$

Following \textbf{Case A}, we have $0\leq r+e\leq 2g<4g+1=d$. Thus, the first condition is equivalent to
$$
r+e\leq g,
$$
or, equivalently, $r\leq g-e$. Similarly, the last condition is equivalent to $r\geq e$. Indeed, if $r<e$, then
$$
\langle r-e\rangle_d=d+r-e>g,
$$
which contradicts the last condition. On the other hand, if $r\geq e$, then
$$
\langle r-e\rangle_d=r-e\leq g.
$$
Therefore, we obtain
$$
e\leq r\leq g-e,
$$
which is equivalent to the above three conditions. This gives the second line of \eqref{P5 Rb-def}. In particular, if $e>\lfloor g/2\rfloor$, the corresponding interval is empty.

The two cases above exhaust all possibilities for $0\leq v<d$, and hence the claim follows.

\begin{lemma} \label{P6 lemma 4}
    For every $\ell\ge0$, we have
\begin{equation}\label{P6 lemma 4 eq1}
  \sum_{j=0}^{d-1} \Theta_{\xi}(s+\ell d + j,s)=0,
\end{equation}
and
\begin{equation}\label{P6 lemma 4 eq2}
  \sum_{j=0}^{d-1} \Xi_{\xi}(s+\ell d + j,s)=0.
\end{equation}
\end{lemma}
\begin{proof}
If $\Psi_v=\varnothing$, then there is nothing to prove, since every term in the corresponding block vanishes by \eqref{P6 T-nonzero-condition 1} and \eqref{P6 T-nonzero-condition 2}.\\
Firstly, suppose that $0\leq v\leq\lfloor g/2\rfloor$. Then the possible non-zero terms in the block are of the form
\begin{equation*}
  n=(u+\ell)d+r,
  \qquad v\leq r\leq g-v.
\end{equation*}

Consider the reflection $r\mapsto r^{\prime}=g-r$, which preserves the interval $[v,g-v]$. If $r=r^{\prime}$, then $2r=g$. Consequently, $8r+1=d$, and hence $[8r+1]_{\xi}=0$. Thus, the corresponding summand vanishes.

Now suppose that $r\neq r^{\prime}$. We compare the term corresponding to $r$ with its reflected counterpart. For $n=(u+\ell)d+r$, the relevant indices are
\begin{equation*}
  n-s=\ell d+(r-v),
  \qquad
  n+s=(2u+\ell)d+(r+v),
  \qquad
  n=(u+\ell)d+r.
\end{equation*}
Again, under the reflection $r\mapsto r^{\prime}=g-r$, we have
\begin{equation*}
  r^{\prime}-v=g-(r+v),
  \qquad
  r^{\prime}+v=g-(r-v).
\end{equation*}
Therefore, when taking the quotient of the reflected term by the original term, the complete-period constants arising from Lemma~\ref{P6 lemma 1} cancel.
The remaining  factors are obtained from Lemma~\ref{P6 lemma 2}:
\begin{align*}
    \frac{H_{r^{\prime}-v}(\xi)H_{r^{\prime}+v}(\xi)H_{r^{\prime}}(\xi)^2}
     {H_{r-v}(\xi)H_{r+v}(\xi) H_r(\xi)^2}
=\xi^{8g^2-4g+12r},
\end{align*}
\begin{align*}
    \frac{H_{r^{\prime}-v}(\xi)H_{r^{\prime}+v}(\xi)H_{r^{\prime}}(\xi)^4}
     {H_{r-v}(\xi)H_{r+v}(\xi) H_r(\xi)^4}
=\xi^{12g^2-6g+18r},
\end{align*}
and
\begin{align*}
\frac{\xi^{2r^{\prime}}}{\xi^{2r}}=\xi^{2g-4r}, \qquad \frac{\xi^{5r^{\prime}}}{\xi^{5r}}=\xi^{5g-10r},
\qquad
\frac{[8r^{\prime}+1]_{\xi}}{[8r+1]_{\xi}}
=-\xi^{-(8r+1)}.
\end{align*}
We have
\begin{align} \label{P6 block-ratio 1}
\frac{\Theta_{\xi}((u+\ell)d+r^{\prime},s)}
     {\Theta_{\xi}((u+\ell)d+r,s)}
&= \xi^{8g^2-4g+12r} \xi^{2g-4r} \left( -\xi^{-(8r+1)} \right)  \notag\\
&=-\xi^{8g^2-2g-1}=-1,
\end{align}
and
\begin{align} \label{P6 block-ratio 2}
\frac{\Xi_{\xi}((u+\ell)d+r^{\prime},s)}
     {\Xi_{\xi}((u+\ell)d+r,s)}
&= \xi^{12g^2-6g+18r} \xi^{5g-10r} \left( -\xi^{-(8r+1)} \right)  \notag\\
&=-\xi^{12g^2-g-1}=-1,
\end{align}
because $$ 8g^2-2g-1 = (2g-1)(4g+1)= (2g-1)d,$$ and $$ 12g^2-g-1=(3g-1)(4g+1)=(3g-1)d.$$
Thus, all non-zero terms cancel in pairs, while any fixed point contributes zero.
Secondly, suppose that $v=d-e$, where $1\leq e\leq\lfloor g/2\rfloor$. Then the potentially non-zero terms in the block are of the form
\begin{equation*}
  n=(u+\ell+1)d+r,
  \qquad e\leq r\leq g-e.
\end{equation*}
Again, the reflection $r^{\prime}=g-r$ preserves this interval, and any fixed point contributes zero. For $r\neq r^{\prime}$, the relevant indices are
\begin{equation*}
  n-s=\ell d+(r+e),
  \qquad
  n+s=(2u+\ell+2)d+(r-e),
  \qquad
  n=(u+\ell+1)d+r.
\end{equation*}
Under the reflection $r\mapsto r^{\prime}$, we have
\begin{equation*}
  r^{\prime}-e=g-(r+e),
  \qquad
  r^{\prime}+e=g-(r-e).
\end{equation*}
The complete-period constants again cancel, and the remaining factors are given by the same calculation as in \eqref{P6 block-ratio 1} and \eqref{P6 block-ratio 2}. Consequently,
\begin{equation*}
  \Theta_\xi((u+\ell+1)d+r^{\prime},s)
  =-\Theta_\xi((u+\ell+1)d+r,s),
\end{equation*}
and
\begin{equation*}
  \Xi_\xi((u+\ell+1)d+r^{\prime},s)
  =-\Xi_\xi((u+\ell+1)d+r,s).
\end{equation*}
Thus, all non-zero terms cancel in pairs, and the proof is complete.
\end{proof}

\section{ Proofs of main results  } \label{Section 3}

We have
\begin{align*}
\Omega_q^1(\mathcal{T}) &= \sum_{n=s}^{\mathcal{T}} \Theta_q(n,s),\\
\Omega_q^2(\mathcal{T}) &= \sum_{n=s}^{\mathcal{T}} \Xi_q(n,s).
\end{align*}
\textbf{Claim :}
\begin{equation} \label{P6 claim}
    \Omega_q^1(\mathcal{T})\equiv 0\pmod{[m]}
    \qquad\text{and}\qquad
    \Omega_q^2(\mathcal{T})\equiv 0\pmod{[m]}.
\end{equation}
To prove \eqref{P6 claim}, it is enough to show that

$$
\Omega_{\xi}^1(\mathcal{T})=0
\qquad\text{and}\qquad
\Omega_{\xi}^2(\mathcal{T})=0,
$$
for every primitive $d$-th root of unity $\xi$, where $d\mid m$ and $d>1$. By Lemma~\ref{P6 lemma 1}, each summand admits a finite limit at $\xi$. Furthermore, Lemma~\ref{P6 lemma 3} guarantees that the reduced denominators of $\Omega_q^1(\mathcal{T})$ and $\Omega_q^2(\mathcal{T})$ are not divisible by $\Phi_d(q)$. Consequently, if both $\Omega_{\xi}^1(\mathcal{T})$ and $\Omega_{\xi}^2(\mathcal{T})$ vanish for every primitive $d$-th root of unity $\xi$, then the reduced numerators of $\Omega_q^1(\mathcal{T})$ and $\Omega_q^2(\mathcal{T})$ must be divisible by $\Phi_d(q)$.

Since $d\mid m$, there exists an integer $N$ such that
\begin{equation*}
m=Nd.
\end{equation*}
We retain the notation $s=ud+v$, with $0\leq v<d$, introduced in \eqref{P5 eq 's'}. Put
\begin{equation*}
\mathcal{E}=\mathcal{T}-s.
\end{equation*}
Then
\begin{align*}
\Omega_\xi^1(\mathcal{T})
&=\sum_{j=0}^{\mathcal{E}}\Theta_\xi(s+j,s),\\
\Omega_\xi^2(\mathcal{T})
&=\sum_{j=0}^{\mathcal{E}}\Xi_\xi(s+j,s).
\end{align*}
Furthermore, by Lemma~\ref{P6 lemma 4}, the contribution from every complete block of $d$ consecutive values of $j$ is zero. Writing
\begin{equation*}
\mathcal{E}=Fd+R,\qquad 0\leq R<d,
\end{equation*}
we see that all complete blocks make no contribution. Hence, it remains only to consider the final incomplete block, namely

$$
    0\leq j-Fd\leq R.
$$
Thus, the evaluation of $\Omega_\xi^1(\mathcal{T})$ and $\Omega_\xi^2(\mathcal{T})$ reduces to the analysis of this final block.
We now characterize the non-vanishing condition in terms of the residue class of the summation shift $j$. For a fixed block, write
$$j=\ell d+b,\qquad 0\le b<d.$$
Since $n=s+j$ and $s\equiv v\pmod d$, it follows that
$$ n\equiv v+b\pmod d.$$
Thus, if $r=\langle n\rangle_d$, then
$$    r\equiv v+b\pmod d,$$
or, equivalently,
$$    b=\langle r-v\rangle_d.$$
Consequently, a residue class $b$ of $j$ can contribute a non-zero summand only when $r\in\Psi_v$. In that case, the corresponding residue class is uniquely determined by
$$    b=\langle r-v\rangle_d.$$

Let $\aleph_v$ denote this set of possible residues $b=\langle j\rangle_d$.
This set is obtained from $\Psi_v$ by subtracting $v$ modulo $d$. In particular, if $0\leq v\leq\lfloor g/2\rfloor$, then
$$
  r=v+a,\qquad 0\le a\le g-2v,
$$
and hence $b=\langle r-v\rangle_d=a$.  If $v=d-e$ with
$1\leq e\leq\lfloor g/2\rfloor$, then
$$
  r\in\{e,e+1,\ldots,g-e\}
$$
and
$$
  b=\langle r-v\rangle_d=\langle r+e\rangle_d=r+e,
$$
because $2e\le r+e\leq g<d$.  Therefore, the possible non-zero residues of $j$
are
\begin{equation}\label{P5 Jb-def}
\aleph_v =
\begin{cases}
\{0,1,\ldots,g-2v\},&
0\le v\le \lfloor g/2\rfloor,\\[2mm]
\{2e,2e+1,\ldots,g\},&
e=d-v,\ 1\leq e\leq \lfloor g/2\rfloor,\\[2mm]
\varnothing,&\text{otherwise}.
\end{cases}
\end{equation}
Thus, only those residue classes belonging to $\aleph_v$ can contribute to the final block. If
$$\{0,1,\ldots,R\}\supseteq\aleph_v,$$
then the final block contains the entire non-zero support of a complete block. Hence, its total contribution vanishes by the Lemma~\ref{P6 lemma 4}. On the other hand, if
$$\{0,1,\ldots,R\}\cap\aleph_v=\varnothing,$$
then every summand in the final block is zero and, therefore the contribution of the final block also vanishes. Now suppose that $\mathcal{T}=(m-1)/4$. Then
\begin{align*}
\mathcal{E}
=\frac{m-1}{4}-s=\left(\frac{N-1}{4}-u\right)d+g-v.
\end{align*}
Thus, the final residue bound is
$$R=g-v\qquad\text{if }v\le g,$$
while, if $v=d-e>g$, we have
$$R=g+e.$$
If $v\le g$, then 
$$R=g-v\ge g-2v,$$
so that the prefix $\{0,1,\ldots,R\}$ contains the entire first set in \eqref{P5 Jb-def}. If, on the other hand, $v=d-e>g$, then
$$R=g+e\ge g,$$
and hence the prefix contains the entire second set in \eqref{P5 Jb-def}. Therefore, in both cases,   the final incomplete block contributes zero.
\par Therefore,
$$\Omega_{\xi}^1(\mathcal{T})=0
\qquad\text{and}\qquad
\Omega_{\xi}^2(\mathcal{T})=0, $$
for every primitive $d$-th root of unity $\xi$, where $d\mid m$ and $d>1$. Consequently, for each such $d$, the cyclotomic polynomial $\Phi_d(q)$ divides the numerator of the reduced form of both $\Omega_q^1(\mathcal{T})$ and $\Omega_q^2(\mathcal{T})$, while their corresponding denominators are relatively prime to $\Phi_d(q)$.
Since cyclotomic polynomials are pairwise coprime and
\begin{equation*}
[m]=\prod_{\substack{d\mid m\ d>1}}\Phi_d(q),
\end{equation*}
it follows that $[m]$ divides the reduced numerators of both $\Omega_q^1(\mathcal{T})$ and $\Omega_q^2(\mathcal{T})$. Moreover, their reduced denominators are relatively prime to $[m]$. This completes the proof of our claim \eqref{P6 claim}.

\begin{proof}[Proof of Theorems \ref{P6 Main Thm 1} and \ref{P6 Main Thm 2}]
We have just shown that
\begin{align} \label{P6 laststep}
    \Omega_q^1(\mathcal{T}) \equiv 0 \pmod{[m]}\quad \text{and} \quad \Omega_q^2(\mathcal{T}) \equiv 0 \pmod{[m]}.
\end{align}
Since $[m+4s](q;q^4)_{\frac{m-1}{4}+s} = [m](q;q^4)_{\frac{m-1}{4}} (q^{m+4};q^4)_s$ is congruent to $0$ modulo $[m].$ By combining \eqref{P6 laststep} with Theorems \ref{Tang G2 Thm 1} and \ref{WangXu Thm}, we complete the proof because least common multiple of $\Phi_m(q)^3$ (or $\Phi_m(q)^5$) with $[m]$ are $[m]\Phi_m(q)^2$ (or $[m]\Phi_m(q)^4$), respectively.
\end{proof}

\section*{Declarations}
	
	{\bf Author contributions:} All authors have contributed equally to the preparation of this manuscript.\\
{\bf Funding:} No funding was received for this research.\\
{\bf Conflicts of Interest:} The authors do not have any conflicts
of interest.
	
	{\bf Data Availability Statement:}
	No datasets were generated or analyzed in this study. The results are obtained entirely through theoretical and mathematical methods.

\end{document}